\documentclass[12pt,reqno]{amsart}
\usepackage{amsfonts} \usepackage{amsmath} 
\usepackage{graphicx}

\theoremstyle{plain}    
  \newtheorem{Thm}{Theorem} 
	\newtheorem{Lem}{Lemma}      
	\newtheorem{Cor}{Corollary}
  \newtheorem{Prop}{Proposition}        
\theoremstyle{definition}
  \newtheorem{Def}{Definition}              
\theoremstyle{remark}

\def \rootedmeet {T_1\cap_r T_2}

\begin{document}
\author{Daniel Selahi Durusoy} 
\title{Crowell's state space is connected}
\address{Michigan State University}
\address{durusoyd@math.msu.edu}

\begin{abstract}We study the set of Crowell states for alternating 
knot projections and show that for prime alternating knots the space of states for a reduced projection is connected, a result similar to that for Kauffman states.  As an application we give a new proof of a result of Ozsv\'ath and Szab\'o characterizing $(2,2n+1)$ torus knots among alternating knots.
\end{abstract}


\maketitle \setlength{\unitlength}{.62mm} 

\section{Introduction}

\noindent
One of the many definitions of the Alexander polynomial of a knot is through state sums. 
Kauffman has described and studied a state sum model for the Alexander polynomial
in great detail \cite{Kauffman}. 
In an earlier paper Crowell has described another state sum model
for the Alexander polynomial for the subclass of alternating knots
(\cite{Crowell59}, Theorem 2.12). 

In the next section we will recall the definition of Crowell states 
 and examine some of their properties. 
In Section~\ref{sec:proof} we will prove that

\begin{Thm} \label{thm:connected}
If $K$ is an alternating prime knot and $D$ is a reduced knot diagram for $K$, 
then any two states differ by a finite sequence of terminal edge exchanges.
\end{Thm}

This theorem is similar in nature to the Clock Theorem of Kauffman \cite{Kauffman} 
which states that any two Kauffman states differ by a finite sequence of clockwise and counterclockwise moves, 
which was also proven in the language of graphs in section 4 of \cite{GilmerLit86}. 
This work is independent of those mentioned 
 because of the simple reason that Kauffman states 
 and Crowell states do not correspond to each other in any natural way
 as observed from the fact that the space of Crowell states do not form a lattice in general (see Proposition~\ref{prop:notlattice}).

In section \ref{sec:application}, as an application we will give an alternative proof that $(2,2n+1)$ torus knots are characterized by their Alexander polynomials among alternating knots, which was originally proven by Ozsv\'ath and Szab\'o (Proposition 4.1 in \cite{OzSz}).

\section{A state model} \label{sec:states}

In this section we will review the definition of the state sum model 
for alternating links given by Crowell \cite{Crowell59} and investigate some properties of the states.

Given a knot $K$ and an oriented alternating diagram $D$ of $K$
with $n$ crossings we obtain 
a weighted labeled directed planar graph $G(D)$ as follows: 
replace a small neighborhood of each crossing by a degree $4$ vertex according to the following figure ($k$ is the vertex label):

%
%
\begin{center} 
\begin{figure}[ht]
\includegraphics{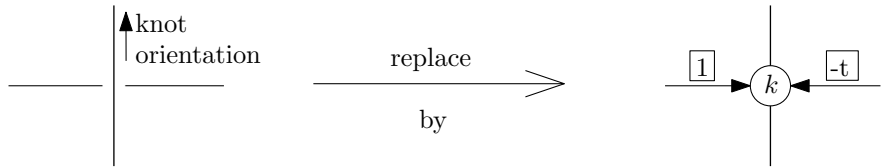} 
\caption{ From a knot diagram to a directed graph. }
\end{figure}
\end{center}

\begin{Prop}\label{prop:coloring} This definition orients all edges and
these orientations are compatible with orientations coming from a checkerboard coloring of the regions in the complement of $D$.
\end{Prop}

\begin{proof}
Each edge gets an orientation since while traveling around a region the strands traveled alternate between under strand and over strand. Hence at each crossing if an edge is coming in, the boundary of the region will continue along the over strand, which  becomes an under strand at its other end, hence we get an orientation on that edge as well, consistent with the previous edge. These orientations are compatible with a checkerboard coloring since crossing a region to another across an edge we get opposite orientations in the plane. 
\end{proof}

%
%

Choose a vertex $v_0$ of $G(D)$. Spanning trees rooted at $v_0$ (edges are directed away from $v_0$) will be called states. Let $Tr(v_0)$ be the space of states and $w(T)$ be the product of weights of all edges in a state $T$. According to \cite[Theorem 2.12]{Crowell59} we get the renormalized Alexander polynomial as a sum of monomials corresponding to each state by
\begin{equation}
 \displaystyle \Delta_K(t) = (-t)^m \cdot \sum_{T\in Tr(v_0)}w(T)
 \label{eq:statesum}
\end{equation}
where $m$ is chosen so that the term with the least power of $t$ is a positive constant.

\begin{Prop}\label{prop:hamiltonpath}
 For any vertex $v$, there is a rooted Hamiltonian path from $v_0$ to $v$ in $G(K)$. 
\end{Prop}

\begin{proof}
 Since $K$ is a knot, $G(K)$ as an unoriented graph is connected. 
 Pick an unoriented path $e_1, e_2, \ldots, e_m$ starting at $v_0$, ending at $v$. 
 If $e_i$ is not oriented away from the root,
   pick edges $e_{i,1}, e_{i,2}, \ldots e_{i,m_i}$ that go around one of the two regions adjacent to $e_i$. 
 Due to Proposition~\ref{prop:coloring}, we get compatible orientations. 
 At the end we get a rooted path which might visit some vertices more than once. 
 For each vertex that is visited more than once, remove all edges between the first and last visits.
\end{proof}

\begin{Cor} \label{cor:extendtree}
Any rooted tree $T$ can be extended to a rooted spanning tree $\tilde{T}$.
\end{Cor}

\begin{proof}
For any vertex $v$ not in $T$, find a rooted Hamiltonian path $\alpha$ 
 from $v_0$ to $v$ using Proposition~\ref{prop:hamiltonpath}.
Add enough of the final segment of $\alpha$ to the current tree so that
 the union will be connected.
\end{proof}

\begin{Prop}\label{prop:statewithprescribedterminalvertex} 
 Given a reduced alternating diagram for a prime knot, any edge in $G(D)$
 except those ending at $v_0$ appear as a teminal edge in at least one state.
\end{Prop}

\begin{proof}
Call the given edge $e_0$ starting at vertex $w_0$, ending at vertex $w_1$. 
We will find a directed Hamiltonian path from $v_0$ to $w_0$, 
 add the edge $e_0$, 
 and extend this path into a rooted spanning tree.

Use Proposition~\ref{prop:hamiltonpath} to construct a rooted Hamiltonian path $\Gamma$ from $v_0$ to $w_0$. There are two possible cases:

\noindent (a) $\Gamma$ doesn't go through $w_1$: 
Do nothing extra. 

\noindent (b) $\Gamma$ goes through $w_1$ before reaching $w_0$: 
Adding $e_0$ to $\Gamma$ produces a loop. 
To avoid this problem we will go around as follows. 
Assume $\Gamma$ contains an edge $e'$ going from $v'$ to $w_1$ and an edge $e''$ going from $w_1$ to $v''$. 
We need to connect $v'$ to $v''$ by an oriented path avoiding $e'$ and $e''$. 
To achieve this, let $R$ be the region bounded on two sides by $e'$ and $e''$ 
 and $\hat{R}$ be the union of regions adjacent to $R$ along edges other than $e'$ and $e''$. 
Then starting at $v'$, following edges on the boundary of  $\hat{R}$ that are not 
 on the boundary of $R$, we reach $v''$. 
Assume that $K$ is prime and $D$ is reduced, then
 this new path $\alpha$ does not include $e'$ and $e''$ since otherwise
 we could draw a separating circle passing through the $v'$ or $v''$ and another common edge of $R$ and $\hat{R}$.
Replacing $e'$ and $e''$ by $\alpha$ in $\Gamma$,
 we get a rooted path from $v_0$ to $w_0$ avoiding $w_1$. 
It could include loops, which can be eliminated as in the proof of Proposition~\ref{prop:hamiltonpath}.

Next we need to extend the rooted Hamiltonian path $\Gamma \cup e_0$ 
 to a rooted spanning tree, keeping $e_0$ a terminal edge.
Consider the two edges coming out from $w_1$, call them $e_1'$ and $e_1''$,
 with terminal vertices $w_1'$ and $w_1''$. 
Using a similar argument as in the proof of case (b) above, 
 consider $R$ being the region bounded by $e_0$ and $e_1'$,
 and $\hat{R}$ the union of regions adjacent to $R$ except along $e_0$ and $e_1'$,
 and removing loops we get a directed path $\beta$ from $w_1$ to $w_1'$. 
 Starting at $w_1'$, add enough edges from the final segment of $\beta$
 so that $w_1'$ is connected to a vertex in $\Gamma \cup e_0$.
Similarly do so for $w_1''$.

Now, use Corollary~\ref{cor:extendtree} to extend this tree to a rooted spanning tree. 
During this process $e_0$ stays a terminal edge since adding $e_1'$ or $e_1''$ would create a loop.
\end{proof}

The state sum in Equation~\ref{eq:statesum} resembles the state sum defined by Kauffman \cite{Kauffman}. Kauffman has studied an operation called {\it clock move} that transforms a state to another that differ only at two crossings and showed that all states differ from one another by a sequence of clock moves. 
With that in mind we define the following operation for reduced alternating diagrams:
	
\begin{Def}  
A state $T_2$ is obtained from a state $T_1$ by a {\em terminal edge exchange} move (edge exchange for short) if replacing a terminal edge in $T_1$ by the other incoming edge at the terminal vertex gives $T_2$. 
\end{Def}

\begin{center} 
\begin{figure}[ht]
 \includegraphics{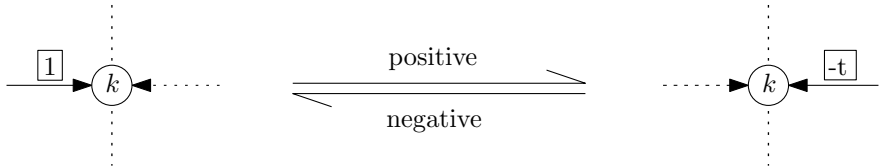} 
 \caption{ Terminal edge exchange. }
\end{figure}
\end{center}

\begin{Prop}\label{prop:terminaledges} 
At any terminal edge, edge exchange gives a new state, except at edges ending at a kink.
\end{Prop}

\begin{proof} If there is a kink at $v$, there is a unique edge that  connects to $v$ in any spanning tree, since the opposite edge is a loop. At any other vertex, it is easy to check that one still gets a rooted spanning tree.
\end{proof}

Edge exchange gives a partial order on the set of states by 
 defining the covering relation of the partial order as 
 $T_1$ comes immediately before $T_2$ if $T_2$ is obtained from $T_1$ by one positive edge exchange. 

Comparing these states with the black trees in Kauffman states, 
 even though states are rooted spanning trees in both models, 
 in Kauffman states the orientations on the edges are chosen after a spanning tree of the black graph is chosen, so the same edge can inherit different orientations in different states. 
Furthermore, consider the graph whose 
 vertices are Crowell states and 
 any two vertices are connected by an edge if there is a terminal edge exchange that takes one state to the other. 
The following proposition shows that edge exchanges do not correspond to clock moves
 under any bijection between the Crowell and Kauffman states 
 since Kauffman states form a distributive lattice \cite{Kauffman}.

\begin{Prop}\label{prop:notlattice}
 The space of Crowell states is not a lattice in general with any choice of a partial order compatible with terminal edge exchanges. 
\end{Prop}

\begin{proof}
Figure~\ref{fig:76states} illustrates the graph of Crowell states for the knot $7_6$ in Rolfsen's table. 
Let us assume that there is a partial order compatible with this graph, 
 i.e., an edge between two states exist if one is an immediate successor of the other. 
Then any degree one vertex is either a local maximum or a local minimum. 
Since this graph has three degree~1 vertices and 
 in a finite lattice there is only one local maximum and only one local minimum, 
 this particular graph can not be a lattice.
\end{proof}
%
%
\begin{center} 
\begin{figure}[ht]
\includegraphics{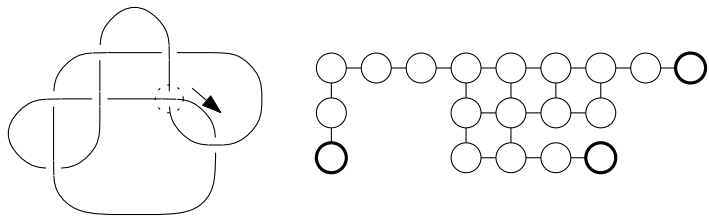}     
\caption{ The knot $7_6$, the chosen root, and its space of states. }  \label{fig:76states}
\end{figure}
\end{center}

%
%
%
%
\section{Proof of Theorem \ref{thm:connected}} \label{sec:proof}
In this section we will assume that $K$ is a prime knot and 
 $D$ is a reduced alternating projection for $K$. 
Choose a root vertex $v_0$ in $G(D)$.
We will provide an algorithm 
to go from one rooted spanning tree $T_1$ to another rooted spanning tree $T_2$ through a sequence of edge exchanges. 
We will label vertices $v$ of $G(D)$ with distinct integers. 

An initial segment $IS(v,T)$ of a rooted spanning tree $T$ is 
 the sequence of vertices on the unique rooted path from the root to $v$ in $T$.
For $v\neq v_0$, let $\phi(v,T)$ denote the vertex that points to $v$ through an edge not in $T$.
Let $Bel(v,T)$ be a small neighborhood of the set of vertices below $v$ in $T$, 
 i.e., those that can be reached from $v$ via directed paths in $T$, 
 the edges between them (not necessarily in $T$) and the elementary regions surrounded by those edges. 
Let $Bel_1(w,T)$ be the connected component of $Bel(w,T)$ containing the successor of $w$ in $T$ with the smaller label.
When the tree $T$ is obvious from the context, we will suppress $T$ from these notations.
The rooted meet of two rooted trees $T_1$ and $T_2$ 
  is the connected component of the root in $T_1 \cap T_2$ and will be denoted by $\rootedmeet$. 

%
%
\begin{center} 
\begin{figure}[ht]
\includegraphics{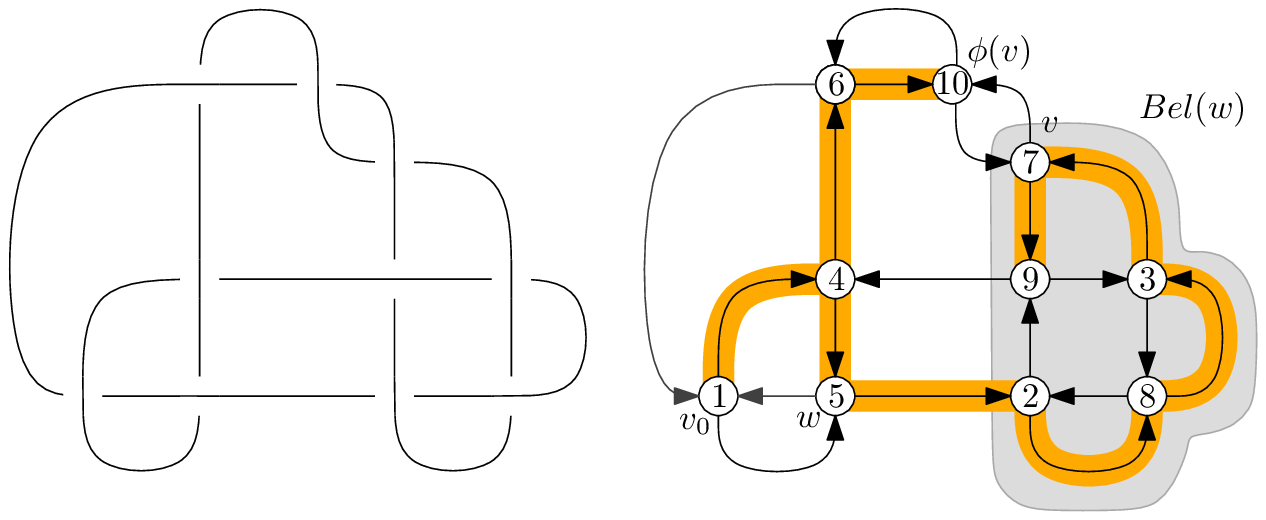}      \label{fig:belw}
\caption{ The knot diagram $D$ and a state $T$ in $G(D)$. }
\end{figure}
\end{center}

In order to prove Theorem~\ref{thm:connected}, we will show that for any given two states $T_1$ and $T_2$, 
  we can persistently enlarge $\rootedmeet$. 

%
%
\begin{Lem}\label{lem:removebeloww} 
Given a state $T$ and a vertex $w$ other than the root $v_0$, 
if $\phi(w,T)\notin Bel(w,T)$,   
then
there is a sequence of edge exchanges 
  that converts the incoming edge for $w$ into a terminal edge, 
  removing edges only below $w$.
\end{Lem}

\begin{Lem}\label{lem:findw'}
Under the conditions of Lemma \ref{lem:removebeloww}, there is a vertex $w'\in Bel(w,T)$ with $\phi(w')\notin Bel(w,T)$.
\end{Lem}

%
\begin{center} 
\begin{figure}[ht]
\includegraphics{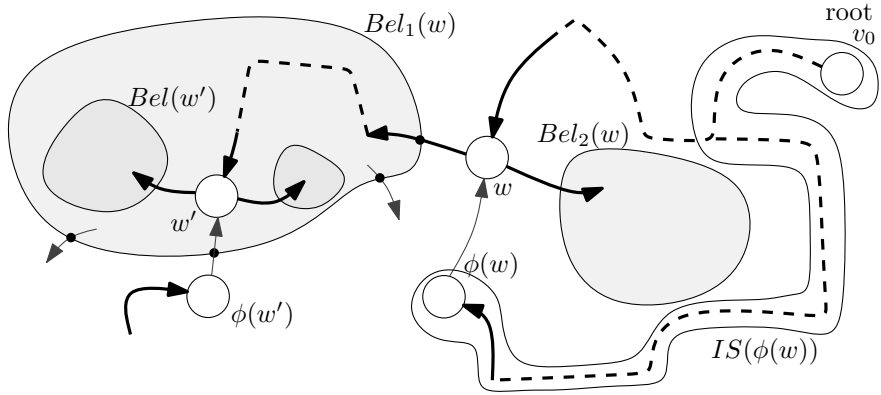}   
\caption{ Thick edges belong to the spanning tree $T$. }  \label{fig:findw'}
\end{figure}
\end{center}

%
%
\begin{proof}[Proof of Lemma~\ref{lem:findw'}]
Assume $Bel(w)$ is nonempty. Given $w$, consider $D \cap \partial Bel_1(w)$. 
Since $K$ is prime and $D$ is reduced, 
  $\partial Bel_1(w)$ is not a separating circle, 
  hence there are at least $4$ intersections. 
Since the orientations of adjacent regions alternate, 
  we get at least two edges entering into $Bel_1(w)$.

Pick a vertex $w'$ among all terminal vertices of edges in $G$ entering into $Bel_1(w)$ not originating from $w$.
This choice implies that $\phi(w')\notin Bel_1(w)$. 
Furthermore the union $IS(w) \cup \{ w \rightarrow \phi(w) \} \cup IS(\phi(w))$ 
  contains an unoriented circuit of edges and vertices 
	  that separate $Bel_1(w)$ from $Bel_2(w)$ (see Figure~\ref{fig:findw'}),
	hence going from $w'$ to a vertex in $Bel_2(w)$ would take at least two edge exchanges.
	We conclude that $\phi(w')\notin Bel_2(w)$ as well. 
\end{proof}

\begin{proof}[Proof of Lemma~\ref{lem:removebeloww}]
If $Bel(w)$ is empty, then $w$ is already a terminal vertex. 
Otherwise, we will use induction on the depth $d$ of the tree $Bel(w)$.

For $d=1$, $Bel(w)$ could contain up to two vertices. 
If there is only one vertex, it is a terminal vertex and an edge exchange empties $Bel(w)$. 
If there are two vertices, $Bel(w)$ has two components, 
  which as in the proof of Lemma~\ref{lem:findw'}, are not adjacent to one another. 
Lemma~\ref{lem:findw'} tells that an edge exchange at either vertex decreases the size of $Bel(w)$, 
 and we are led to the case of one vertex.

Assume the hypothesis is true for all trees $Bel(w)$ of depth $d$ and less.
If $Bel(w)$ has depth $d+1$, use Lemma~\ref{lem:findw'} to find a $w'$
 with the property $\phi(w')\notin Bel(w)$, in particular $\phi(w')\notin Bel(w')$.
Therefore by the induction hypothesis $w'$ becomes a terminal vertex
 after a finite sequence of edge exchanges only removing edges in $Bel(w')$.
Then performing an edge exchange at $w'$ decreases the size of $Bel(w)$.
Hence repeating this process $w$ becomes a terminal vertex
 while only edges below $w$ being removed throughout the process.
\end{proof}

\begin{proof}[Proof of Theorem \ref{thm:connected}] Given two distinct rooted spanning trees $T_1$ and $T_2$,
pick a vertex $w$ adjacent to $\rootedmeet$ along an edge in $T_2$.
Note that $w\neq v_0$ since $v_0\in \rootedmeet$ 
 and once $w$ is a terminal edge, 
 after an edge exchange $w$ will be connected  to $v_0$ 
 along the same rooted Hamiltonian path as in $T_2$.

Let $v_1, v_2$ be the vertices that lead to $w$ in $T_1$ and $T_2$ respectively. 
By definition, $v_2=\phi(w,T_1)$ and $v_2\in \rootedmeet$. 
Hence $IS(\phi(w,T_1),T_1)\subset T_1\cap_r T_2$, 
  but $w\notin \rootedmeet$, 
	hence $w \notin IS(\phi(w,T_1),T_1)$, 
  which means $\phi(w,T_1)\notin Bel(w,T_1)$.

Applying Lemma \ref{lem:removebeloww}, 
  we get a sequence of edge exchanges that ends in a state where $w$ is a terminal vertex 
  without removing any edges from the rooted meet. 
Now perform an edge exchange at $w$, this enlarges the rooted meet. 
Since the rooted meet only enlarges during this process, 
 in finitely many repetitions of this process we reach $T_2$.
\end{proof}

%
%
%
%
\section{An application to $(2,2n+1)$ torus knots} \label{sec:application}

In this section we will provide a different proof of the following result originally proved by Ozsv\'ath and Szab\'o:
\begin{Thm} The $(2,2n+1)$ torus knots are characterized among alternating knots by the Alexander polynomial.
\end{Thm}

\proof
Let $D$ be a reduced alternating projection for a knot $K$ with $\Delta_K(t)=1+(-t)+(-t)^2+...+(-t)^{2n}$. 
Since all coefficients of powers of $-t$ are $+1$, each state has a different weight and
 $\Delta_K(t)$ is not a product of two alternating knot polynomials (c.f. \cite[Prop. 4.1]{OzSz}), hence $K$ is prime. 

Let $T_0$ be the state with the least $t$ power. 
Since $K$ is prime and the fact that each edge exchange changes the power of $-t$ by $\pm1$, 
 using Theorem~\ref{thm:connected}
 we get a linear ordering on the $2n+1$ states starting at $T_0$,  
 reaching each next state by exchanging an edge of weight $+1$ with an edge of weight $-t$. 
 
According to Proposition~\ref{prop:terminaledges} and due to this linear order, 
 $T_0$ and the top state $T_{2n}$ have only one terminal edge each,
 hence they have no branching, whereas
 intermediate states have $2$ terminal edges.

Since Proposition~\ref{prop:statewithprescribedterminalvertex} tells that each edge $v$
  (except the two that point to $v_0$)
  can be extended to a state having $v$ as a terminal edge, 
  and since we can reach that state from $T_0$ by positive edge exchanges, 
	we see that all edges in $T_0$ have weight $+1$.
We conclude that $T_0$ has $2n$ edges 
 since each edge of weight $+1$ is used only once in an edge exchange
 and no new edges emerge with weight $+1$ as we go from $T_0$ to $T_{2n}$.

Edge orientations and weights do not depend on the choice of the root vertex, 
  hence, after moving the root from $v_0$ to $v_1$,
	we still get a space of states with the same properties, in particular, 
	there will be a new state $T'_0$ containing a linear directed chain of $2n$ vertices starting
	at $v_1$, ending at $v_0$. Hence we get a cycle of length $2n+1$ of edges of weight $+1$.
Similarly, all remaining edges have weight $-t$, form a loop and are used in $T_{2n}$, 
	except the one pointing at the root.

\begin{center} 
\begin{figure}[ht]
\includegraphics{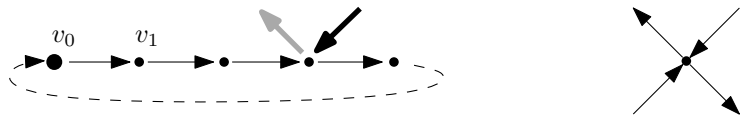} 
\caption{ Thick edges have weight $-t$. On the right, orientations at a typical node. }
\end{figure}
\end{center}

This information tells us that if there is an incoming edge of weight $-t$ at a vertex $v$,
 the next edge of weight $-t$ has to be on the same side of the loop of $+1$ edges due to the cyclic
 alternating orientation of edges at a vertex.
Since these edges with weight $-t$ form a loop as well, 
 they have to go between consecutive vertices. 
This gives us the diagram for the $(2,2n+1)$ torus knot.
\qed

\vspace{2mm}

\noindent {\em Acknowledgements.} 
I would like to thank Bedia Akyar for the invitation to give a talk
at Dokuz Eyl\"ul University, during which time period I started exploring the properties of the Crowell state space.
Most of the work was done during my time at Ferris State University. I would also like to thank Mahir Bilen Can
and Mohan Bhupal for providing feedback.
 \end{document}